\documentclass{amsart}
\usepackage{amssymb,amsmath,latexsym}
\usepackage{amsthm}
\usepackage{fontenc}
\usepackage{amssymb}
\numberwithin{equation}{section}

\newtheorem{theorem}{Theorem}[section]

\begin{document}
\author{Alexander E Patkowski}
\title{A New Integral Equation and Some Integrals Associated with Number Theory}

\maketitle
\begin{abstract}We utilize a combination of integral transforms, including the Laplace transform, with some classical results in analytic number theory concerning the Riemann $\xi$-function, to obtain a new integral equation. We also provide a new proof of known functional-type identities from analytic number theory, and recast some criteria associated with the RH. We also describe an application of our integral equation to the Dirichlet problem in the half plane, giving a new application of the Riemann xi function.\end{abstract}

\keywords{\it Keywords: \rm Fourier Integrals; Riemann Xi Function}

\subjclass{ \it 2010 Mathematics Subject Classification 11M06, 33C05.}

\section{Introduction}
In Titchmarsh [16, pg.35], we find a famous relation connecting Fourier cosine transforms with Mellin transforms 
\begin{equation}\int_{0}^{\infty}f(t)\lambda(t)\cos(xt)dt=\frac{1}{2i\sqrt{y}}\int_{\frac{1}{2}-i\infty}^{\frac{1}{2}+i\infty}\phi(s-\frac{1}{2})\phi(\frac{1}{2}-s)(s-1)\Gamma(1+\frac{s}{2})\pi^{-\frac{s}{2}}\zeta(s)y^sds,\end{equation}
where $f(t)=|\phi(it)|^2$ ($\phi(s)$ is analytic). $\zeta(s)$ is the Riemann zeta function [8, 16], $\Gamma(s)$ is the gamma function, $\xi(s):=\frac{1}{2}s(s-1)\pi^{-\frac{s}{2}}\Gamma(\frac{s}{2})\zeta(s),$ $\lambda(t):=\xi(\frac{1}{2}+it)$ (Riemann's $\Xi$ function), and $y=e^{x}.$ Recall [8, 16] the famous Riemann $\xi$ function satisfies the functional equation
\begin{equation}\xi(s)=\frac{1}{2}s(s-1)\pi^{-\frac{s}{2}}\Gamma(\frac{s}{2})\zeta(s)=\frac{1}{2}s(s-1)\pi^{-\frac{1-s}{2}}\Gamma(\frac{1-s}{2})\zeta(1-s).\end{equation} 
We shall let $\Re(s)$ and $\Im(s)$ denote the real and complex parts, respectively, for $s\in\mathbb{C}.$ Many authors [5, 6, 10, 11, 14, 15, 16] have utilized equation (1.1) and its variations to obtain many interesting relations among other special functions, as well as results on $\lambda(t).$ See also [12] for some interesting ideas on integrals related to Riemann's $\xi$ function. The most widely cited example in the literature appears to be [16, eq.(2.16.1)] (for example, see [2, 8]), and its variant [16, eq.(2.16.2)]:
\begin{equation}\Lambda(x):=\int_{0}^{\infty}\frac{\lambda(t)}{t^2+\frac{1}{4}}\cos(xt)dt=\frac{\pi}{2}\left(e^{x/2}-2e^{-x/2}\psi(e^{-2x})\right).\end{equation}
Here we have a slightly modified Jacobi theta function $\psi(x)=\sum_{n\ge1}e^{-\pi n^2 x}.$ \par The purpose of this paper is to offer some further consequences of the integral formula (1.3) that appear to be overlooked. The main one being a new integral equation that has some resemblance to the Fredholm integral equation of the second kind, which may be 
helpful in studying the zero's of the Riemann $\xi$ function. The integral equation implies several known important facts, including the integral representation
for the Riemann xi-function, and a convergent series involving the incomplete gamma function. We believe it is most
likely that, developing properties of its solutions will lead to further information about the Riemann $\xi$ function.
\begin{theorem} \it When $\Re(s)>1,$ we have
\begin{equation}\begin{aligned}&\Upsilon(s):=(s-\frac{1}{2})\int_{0}^{\infty}\frac{\lambda(t)}{(t^2+\frac{1}{4})(t^2+(s-\frac{1}{2})^2)}dt\\
&=\frac{\pi}{2}\left(\frac{1}{s-1}-\frac{2\xi(s)}{s(s-1)}+\pi^{-s/2}\sum_{n\ge1}n^{-s}\Gamma(\frac{s}{2},\pi n^2)\right).\end{aligned}\end{equation}
Or equivalently, the Riemann $\xi$-function satisfies the integral equation 
\begin{equation}\int_{0}^{\infty}f(\frac{1}{2}+it)K(s,t)dt=\frac{\pi}{2}\left(\frac{-2f(s)}{s(s-1)}+g(s)\right),\end{equation} with kernal $K(s,t)=(t^2+\frac{1}{4})^{-1}(t^2+(s-\frac{1}{2})^2)^{-1}(s-\frac{1}{2}).$ 
\end{theorem}
\begin{proof} As usual, we define the Laplace transform to be
\begin{equation}\mathcal{L}(f)(s):=\int_{0}^{\infty}f(t)e^{-st}dt.\end{equation}
If we assume $\Re(s)>1,$ we may write
\begin{align}\mathcal{L}(\Lambda(x))(s-\frac{1}{2})\\
&=\int_{0}^{\infty}e^{-(s-\frac{1}{2})x}\int_{0}^{\infty}\frac{\lambda(t)}{t^2+\frac{1}{4}}\cos(xt)dtdx\\
&=(s-\frac{1}{2})\int_{0}^{\infty}\frac{\lambda(t)}{(t^2+\frac{1}{4})(t^2+(s-\frac{1}{2})^2)}dt\\
&=\frac{\pi}{2}\int_{0}^{\infty}e^{-sx}\left(e^{x}-2\psi(e^{-2x})\right)dx\\
&=\frac{\pi}{2}\left(\frac{1}{s-1}-2\int_{0}^{\infty}e^{-sx}\psi(e^{-2x})\right)dx\\
&=\frac{\pi}{2}\left(\frac{1}{s-1}+2\lim_{r\rightarrow0^{+}}\int_{1}^{r}t^{s-1}\psi(t^2)dt\right)\\
&=\frac{\pi}{2}\left(\frac{1}{s-1}-x^{-\frac{s}{2}}\sum_{n\ge1}\lim_{r\rightarrow0^{+}}\int_{r}^{x}t^{\frac{s}{2}-1}e^{-\pi n^2 t/x}dt\right)\\
&=\frac{\pi}{2}\left(\frac{1}{s-1}-\pi^{-\frac{s}{2}}\sum_{n\ge1}n^{-s}\gamma(\frac{s}{2},\pi n^2)\right)\\
&=\frac{\pi}{2}\left(\frac{1}{s-1}-\pi^{-\frac{s}{2}}\sum_{n\ge1}n^{-s}\left(\Gamma(\frac{s}{2})-\Gamma(\frac{s}{2},\pi n^2)\right)\right)\\
&=\frac{\pi}{2}\left(\frac{1}{s-1}-\frac{2\xi(s)}{s(s-1)}+\pi^{-s/2}\sum_{n\ge1}n^{-s}\Gamma(\frac{s}{2},\pi n^2)\right)
\end{align}
In the line (1.13) we made a change of variables $x=-\log t,$ and in subsequent lines employed the definition of the incomplete gamma functions
$$\gamma(s,x)=\int_{0}^{x}t^{s-1}e^{-t}dt, ~~~~~~~~\mbox{and}~~~~~~~~ \Gamma(s,x)=\int_{x}^{\infty}t^{s-1}e^{-t}dt,$$
where $\Gamma(s)=\gamma(s,x)+\Gamma(s,x).$ The assertion that $g(s)$ may be explicitly computed follows directly from properties of Mellin transforms. \end{proof}
Note that by an instance $F(y)=e^{-ay^2}$ of the M\"{u}ntz formua [15, eq.(2.11.1)], for $0<\sigma<1,$
\begin{equation}\frac{1}{2\pi i}\int_{\sigma-i\infty}^{\sigma+i\infty}\Gamma(\frac{s}{2})\zeta(s)(\sqrt{a}x)^{-s}ds=\sum_{n\ge1}e^{-an^2x^2}-\frac{1}{2x}\sqrt{\frac{\pi}{a}}.\end{equation}
Multiplying through by $x^{v-1},$ with $\Re(v)>1,$ and integrating over the interval $[0, z],$ $z>0,$ we obtain
\begin{equation}\frac{z^v}{2\pi i}\int_{\sigma-i\infty}^{\sigma+i\infty}\Gamma(\frac{s}{2})\zeta(s)\frac{(\sqrt{a}z)^{-s}}{(v-s)}ds=\frac{a^{-v/2}}{2}\sum_{n\ge1}n^{-v}\gamma(\frac{v}{2},a zn^2)-\frac{z^{v-1}}{2(v-1)}\sqrt{\frac{\pi}{a}},\end{equation}
again for $0<\sigma<1.$ (Note this integration is justified on the left side, since if $\Re(v-s)>0,$ then $0^{v-s}=0.$) On the other hand the integral on the left-hand side of (1.18) is precisely $\Upsilon(v)$ defined in Theorem 1.1, when $a=\pi,$ $z=1,$ $\sigma=\frac{1}{2}.$ Hence we have arrived twice at Theorem 1.1. \par
From Titchmarsh [16, pg. 257] we have $\lambda(t)\ll t^Ae^{-\frac{\pi}{4}t},$ which gives us 
$$\Upsilon(s)\ll \int_{0}^{\infty}\frac{t^Ae^{-\frac{\pi}{4}t}}{(t^2+\frac{1}{4})(t^2+(s-\frac{1}{2})^2)}dt<+\infty.$$
\par Clearly we have that $\Upsilon(s)=-\Upsilon(1-s).$ Upon noticing this fact, we may use (1.18) to obtain a well-known result concerning $\zeta(s).$ Computing the residue $R_{s=0}=\zeta(0)/v=-1/(2v)$ out from the integral, we may then extract the series \newline $\frac{1}{2}\pi^{-(1-v)/2}\sum_{n\ge1}n^{-(1-v)}\Gamma(\frac{1-v}{2},\pi n^2)$ to obtain the following expansion [1, pg.256, eq.(30)] as a direct corollary to Theorem 1.1:
\begin{equation}\pi^{-s/2}\zeta(s)\Gamma(\frac{s}{2})=\frac{1}{s(s-1)}+\pi^{-s/2}\sum_{n\ge1}n^{-s}\Gamma(\frac{s}{2},\pi n^2)+\pi^{-(1-s)/2}\sum_{n\ge1}n^{-(1-s)}\Gamma(\frac{1-s}{2},\pi n^2).\end{equation}

\section{Imaginary quadratic forms and the associated integral equation}
Following [8, pg.511] put
\begin{equation}L_{K}(s, \chi)=\sum_{\mathfrak{a}}\chi(\mathfrak{a})(N\mathfrak{a})^{-s},\end{equation} where $\Re(s)>1,$
and $\chi$ maps the class group $\mathfrak{H}$ to the complex plane $\mathbb{C}.$ In [8, 9], we find N. S. Koshlyakov investigating integrals related to $L$-functions associated with number fields as well as [14, 16]. We consider his work coupled with that of the ideas in the introduction. Let $D$ denote a discriminant with respect to a primitive ideal $\mathfrak{a}.$ Then we have 
the functional equation [8, eq.(22.51)]
\begin{equation} \Omega_K(s,\chi)=\Omega_K(1-s,\chi),\end{equation}
where $\Omega_K(s,\chi)=(2\pi)^{-s}\Gamma(s)|D|^{\frac{s}{2}}L_K(s,\chi).$ \par An important integral representation relevant to our study, which continues $L_K(s,\chi)$ to the entire complex plane, is given by Hecke [8, eq.(22.52)]:
\begin{equation}\Omega_K(s,\chi)=\frac{|\mathfrak{H}|\delta(\chi)}{\bar{D}s(s-1)}+\int_{1}^{\infty}(t^{s-1}+t^{-s})\sum_{\mathfrak{a}}\chi(\mathfrak{a})e^{-2\pi tN\mathfrak{a}/\sqrt{|D|}}dt,\end{equation}
where $\mathfrak{|H|}$ is the class number, and $\bar{D}$ is $6$ if $D=-3,$ $-4$ if $D=4,$ and $2$ if $D<-4.$ We shall prove an equivalent form of this integral representation toward the end using Theorem 2.1. Put  $\Omega_K(\frac{1}{2}+it,\chi)=\mathfrak{O}_K(t),$ and note that
\begin{equation}\int_{0}^{\infty}f(t)\mathfrak{O}_K(t)\cos(xt)dt=\frac{1}{2i\sqrt{y}}\int_{\frac{1}{2}-i\infty}^{\frac{1}{2}+i\infty}\phi(s-\frac{1}{2})\phi(\frac{1}{2}-s)(2\pi)^{-s}|D|^{s/2}\Gamma(s)L_K(s,\chi)y^sds.\end{equation}
N.S. Koshlyakov appears to be the first to consider instances of this type of general integral with $L$-functions associated with number fields (see, especially, [11, pg.217--220]). In this section we give an equivalent integral equation for imaginary quadratic forms, using all these ideas.

\begin{theorem} \it For $\Re(s)>1,$ we have that
\begin{equation}\mathfrak{\Upsilon}(s)=\int_{0}^{\infty}\mathfrak{O}_K(t)\bar{K}(s,t)dt=\frac{\pi}{2}\left(\frac{|\mathfrak{H}|\delta(\chi)}{\bar{D}(s-1)}-\Omega_K(s,\chi)+\sum_{\mathfrak{a}}\chi(\mathfrak{a})\left(\frac{\sqrt{|D|}}{2\pi N\mathfrak{a}}\right)^s\Gamma(s,\frac{2\pi N\mathfrak{a}}{\sqrt{|D|}})\right),\end{equation}
with kernel $\bar{K}(s,t)=(s-\frac{1}{2})(t^2+(s-\frac{1}{2})^2)^{-1}.$ 
\end{theorem}
\begin{proof}
We may easily evaluate the case $\phi(s)=1,$ and so $f(t)=1,$ for all $t.$ Consider
\begin{equation}\frac{1}{2\pi i}\int_{d-i\infty}^{d+i\infty}(2\pi)^{-s}|D|^{s/2}\Gamma(s)L_K(s,\chi)y^sds,\end{equation}
for $d>1,$ and move the line of integration to $d=\frac{1}{2}.$ Since $L_K(s,\chi)$ has a simple pole at $s=1$ when $\chi$ is trivial, we compute [8, pg.512, eq.(22.50)] $$R_{s=1}=\frac{|\mathfrak{H}|\delta(\chi)}{\bar{D}},$$ (where $\delta(\chi)=0$ when $\chi$ is non-trivial, $1$ otherwise) and obtain
\begin{equation}\int_{0}^{\infty}\mathfrak{O}_K(t)\cos(xt)dt=\frac{\pi}{2}\left(\frac{|\mathfrak{H}|\delta(\chi)e^{x/2}}{\bar{D}}-e^{-x/2}\Psi(e^{-x})\right),\end{equation}
where $\Psi(x)=\sum_{\mathfrak{a}}\chi(\mathfrak{a})e^{-2\pi(N\mathfrak{a})x/\sqrt{|D|}}.$ (Note that (2.6) is precisely $\Psi(y^{-1}).$)
Applying the same concepts as in our introduction, we may obtain the theorem.

\end{proof}

Due to the kernel, we again have $\mathfrak{\Upsilon}(s)=-\mathfrak{\Upsilon}(1-s),$ and as a direct corollary we have [8, pg. 512, eq.(22.54)]
\begin{equation}\Omega_K(s,\chi)=\frac{|\mathfrak{H}|\delta(\chi)}{\bar{D}s(s-1)}+\sum_{\mathfrak{a}}\chi(\mathfrak{a})\left(\frac{\sqrt{|D|}}{2\pi N\mathfrak{a}}\right)^s\Gamma(s,\frac{2\pi N\mathfrak{a}}{\sqrt{|D|}})+\sum_{\mathfrak{a}}\chi(\mathfrak{a})\left(\frac{\sqrt{|D|}}{2\pi N\mathfrak{a}}\right)^{1-s}\Gamma(1-s,\frac{2\pi N\mathfrak{a}}{\sqrt{|D|}}).\end{equation}

\section{Dirichlet's Problem in the half plane and other Applications}

Here we discuss some results related to (1.3) and (1.5) that we hope will encourage further interest. First we discuss the relationship to the solution of Dirichlet's problem in the half plane,

\begin{equation}\frac{\partial^2 u}{\partial y^2}+\frac{\partial^2 u}{\partial x^2}=0,\end{equation}
where $y\in\mathbb{R},$ $x\ge0,$ initial condition $u(y,0)=h(y),$ and growth condition $u(y,x)\rightarrow0,$ as $|y|\rightarrow\infty.$ The known solution is given as  the Poisson integral [4, pg.36],
$$u(y,x)=\frac{x}{\pi}\int_{\mathbb{R}}\frac{h(t)}{(t-y)^2+x^2}dt.$$
Making the change of variables $t=t+y,$ $x=s-\frac{1}{2},$ we may state this as 
\begin{equation}u(y,s-\frac{1}{2})=\frac{s-\frac{1}{2}}{\pi}\int_{\mathbb{R}}\frac{h(t+y)}{t^2+(s-\frac{1}{2})^2}dt,\end{equation}
Hence, to incorporate our integral equation, we could impose the additional condition to the Dirichlet problem that,
\begin{equation}u(0,s-\frac{1}{2})=\frac{s-\frac{1}{2}}{\pi}\int_{\mathbb{R}}\frac{h(t)}{t^2+(s-\frac{1}{2})^2}dt=h(-i(s-\frac{1}{2}))-r(s),\end{equation}
where $r(s)$ is analytic for $\Re(s)>1.$ Our results are applicable if one sets 
$$r(s)=\frac{C}{s-1}+\pi^{-s/2}\sum_{n\ge1}\frac{a_n}{n^{s}}\Gamma(\frac{s}{2},\pi n^2),$$ where $C$ is a constant, and the coefficients $a_n$ depend on $L(s),$ and subsequently the solution to Theorem 1.1 equates to $h(t)=\gamma(\frac{1}{2}+it)L(\frac{1}{2}+it),$ with the associated gamma factor $\gamma(s).$ The particular case of (1.4), we have $h(t)=\pi^{-(\frac{1}{2}+it)/2}\Gamma(\frac{1}{2}+it)\zeta(\frac{1}{2}+it).$

\par Next we mention other tangential results concerning criteria for the Riemann Hypothesis. First we recall for $0<\Re(s)<1,$ the well-known formula
\begin{equation} \int_{0}^{\infty}t^{s-1}\cos(xt)dt=\frac{\Gamma(s)\cos(\frac{\pi}{2}s)}{x^s}.\end{equation}
Hence for $0<\Re(s)<1,$
\begin{equation} \int_{0}^{\infty}t^{s-1}\left(\int_{0}^{\infty}\cos(xt)(2\psi(x^2)-\frac{1}{x})dx\right)dt=\frac{2\xi(s)\Gamma(s)\cos(\frac{\pi}{2}s)}{s(s-1)}.\end{equation}
This gives us ($0<c<1$)
\begin{equation} 2\psi(x^2)-\frac{1}{x}=\int_{0}^{\infty}\cos(xt)\left(\frac{1}{2\pi i }\int_{c-i\infty}^{c+i\infty}\frac{2\xi(s)\Gamma(s)\cos(\frac{\pi}{2}s)t^{-s}}{s(s-1)}ds\right)dt=\int_{0}^{\infty}\cos(xt)\dot{I}(t)dt,\end{equation}
say.
So we may now compute (using (1.3))
\begin{equation}-\frac{\lambda(t)}{t^2+\frac{1}{4}}=\int_{0}^{\infty}\cos(xt)e^{-x/2}\left(\int_{0}^{\infty}\cos(e^{-x}u)\dot{I}(u)du\right)dx.\end{equation}
It is also possible to prove (3.5) using Parseval's theorem for Mellin transforms and the functional equation (1.2). \par We now recast the known criteria for the RH concerning (3.6) outlined in [3] in a different form. (See [3, Definition 1.1] for a criterion for an entire function to belong to the Laguerre-P\'{o}lya class.)
\begin{theorem} The Riemann Hypothesis is equivalent to the statement that the integral
$$\ddot{I}(t):=\int_{0}^{\infty}\cos(xt)\left(-\partial_x^2+\frac{1}{4}\right)e^{-x/2}\left(\int_{0}^{\infty}\cos(e^{-x}u)\dot{I}(u)du\right)dx,$$
has only real zeros. Consequently, the Riemann Hypothesis is true if and only if $\ddot{I}(t)$ is in the Laguerre-P\'{o}lya class. 
\end{theorem}
\begin{proof} The proof uses the integral (3.7), the operator from [6], coupled with comparing the
kernel of $\ddot{I}(t),$ given by $$\ddot{k}(x):=\left(-\partial_x^2+\frac{1}{4}\right)e^{-x/2}\left(\int_{0}^{\infty}\cos(e^{-x}u)\dot{I}(u)du\right)dx,$$ to what is considered an `admissible' kernel 
according to the definition given in [3, Definition 1.2]. For example, it is easily verified that $\ddot{k}(x)$ is even by the functional equation for the Jacobi theta function (and (3.6)), $\ddot{k}(x)>0$ for all $x\in\mathbb{R},$ and $\frac{d}{dx}\ddot{k}(x)<0$ for all $x>0.$ Hence the kernel is monotone decreasing for $x>0.$ The fact that the kernel satisfies these properties tells us that $\ddot{I}(t)$ is a real entire function, and it is possible to have only real zeros. The remainder of the proof utilizes the well-known criterion that the
Riemann hypothesis is equivalent to the statement that all the zeros of $\lambda(t)$ are real, together with [3, Definition 1.1].
\end{proof}

We can also include the Hankel transform in our re-stating of the RH criteria found in [3] using the above ideas. Recall that the Hankel transform of a suitable function $f(r)$ is given by $H(k)=\int_{0}^{\infty}f(r)J_{v}(kr)rdr,$ and $J_{v}(x)$ is the Bessel function.
\begin{theorem} The Riemann Hypothesis is equivalent to the statement that the function
$$\bar{G}(t):=\int_{0}^{\infty}\cos(xt)\left(-\partial_x^2+\frac{1}{4}\right)e^{-3x/2}\left(\int_{0}^{\infty}J_{0}(e^{-x}u)\dot{H}(u)du\right)dx,$$
has only real zeros. Consequently, the Riemann Hypothesis is true if and only if $\bar{G}(t)$ is in the Laguerre-P\'{o}lya class. 
\end{theorem}
\begin{proof} We mimic the proof of the last theorem. First, for $0<\Re(s)<1$ we have that
$$\int_{0}^{\infty}t^{s-1}\left(\int_{0}^{\infty}J_{0}(kt)(2\psi(k^2)-\frac{1}{k})\right)dt=\frac{2^{s-1}\xi(s)\Gamma(\frac{s}{2})}{\Gamma(1-\frac{s}{2})s(s-1)}.$$
Hence, using the inverse Mellin and inverse Hankel transform (with $v=0$), we find
$$\frac{1}{k}(2\psi(k^2)-\frac{1}{k}))=\int_{0}^{\infty}J_{0}(kt)t\dot{H}(t)dt,$$
where ($c\in(0,1)$)
$$\dot{H}(t)=\frac{1}{2\pi i}\int_{c-i\infty}^{c+i\infty}\frac{t^{-s}2^{s-1}\xi(s)\Gamma(\frac{s}{2})}{\Gamma(1-\frac{s}{2})s(s-1)}ds.$$
Putting $k=e^{-x},$ multiplying by $e^{-3x/2},$ taking the Fourier cosine transform and the applying the same arguments we used previously gives the result.
\end{proof}

\par At this point we are left with some questions regarding the solutions of Theorem 1.1. First, in light of equation (1.6) resembling a ``modified" Fredholm integral of the second kind, does the integral equation in Theorem 1.1 admit application of a kind of Fredholm theory? Can anything be said regarding the uniqueness of solution of (1.6) for each $g(s)$? It is clear through Mellin transforms that uniqueness of $f(s)$ follows for the choices we have made for $g(s).$ However, it is still open if this is true for any analytic $g(s),$ as the general Fredholm theory doesn't clearly offer an answer. \par More interesting relations may be obtained by expanding $L_K(s,\chi)$ in (2.4) into a sum of Epstein zeta functions, with formulas related to the work in [9, 10, 17]. This would allow us to obtain integrals related to a theta functions of the form $\sum_{n,m\in\mathbb{Z} }\chi_{n,m}q^{an^2+bnm+cm^2},$ $4ac-b^2>0,$where $q=e^{-\pi x},$ $x>0.$ It may also be of interest to look at the integral $\ddot{I}(t)$ by explicitly computing $\dot{I}(u)$ using the Residue theorem.

1390 Bumps River Rd. \\*
Centerville, MA
02632 \\*
USA \\*
E-mail: alexpatk@hotmail.com
\end{document}